\documentclass[review,onefignum,onetabnum]{siamart171218}

\usepackage{lipsum}
\usepackage{mathrsfs}
\usepackage{amsfonts}
\usepackage{graphicx}
\usepackage{epstopdf}
\usepackage{algorithmic}
\usepackage{amssymb}
\usepackage{amsmath}
\usepackage{cleveref}
\crefname{lemma}{Lemma}{Lemmas}

\ifpdf
  \DeclareGraphicsExtensions{.eps,.pdf,.png,.jpg}
\else
  \DeclareGraphicsExtensions{.eps}
\fi

\newsiamremark{hypothesis}{Hypothesis}
\crefname{hypothesis}{Hypothesis}{Hypotheses}
\newsiamthm{claim}{Claim}

\nolinenumbers

\headers{Computation of parabolic cylinder functions }{T. M. Dunster, A. Gil, and J. Segura}

\title{Computation of parabolic cylinder functions having complex argument}

\author{T. M. Dunster\thanks{Department of Mathematics and Statistics, San Diego State University, 5500 Campanile Drive, San Diego, CA 92182, USA. 
  (\email{mdunster@sdsu.edu}, \url{https://tmdunster.sdsu.edu}).}
\and A. Gil\thanks{Departamento de Matem\'atica Aplicada y CC. de la Computaci\'on, ETSI Caminos, Universidad de Cantabria, 39005-Santander, Spain. 
  (\email{amparo.gil@unican.es}, \url{https://personales.unican.es/gila}).}
\and J. Segura\thanks{Departamento de Matem\'aticas, Estad\'{\i}stica y Computaci\'on, Facultad de Ciencias, Universidad de Cantabria, 39005-Santander, Spain. 
  (\email{javier.segura@unican.es}, \url{https://personales.unican.es/segurajj}).}}
\usepackage{amsopn}

\makeatletter
\newcommand*{\addFileDependency}[1]{
  \typeout{(#1)}
  \@addtofilelist{#1}
  \IfFileExists{#1}{}{\typeout{No file #1.}}
}
\makeatother

\nolinenumbers

\ifpdf
\hypersetup{
  pdftitle={Computation of parabolic cylinder functions having complex argument},
  pdfauthor={T. M. Dunster, A. Gil, and J. Segura}
}

\begin{document}

\maketitle

\begin{abstract}
 Numerical methods for the computation of the parabolic cylinder $U(a,z)$ for real $a$ and complex $z$ are presented. The 
 main tools are recent asymptotic expansions involving exponential and Airy functions, with slowly varying analytic coefficient functions involving simple coefficients, and stable integral representations; these two main main methods can be complemented with 
 Maclaurin series and a Poincar\'e asymptotic expansion. 
 We provide numerical evidence showing that the combination of these methods is enough for 
 computing the function with $5\times 10^{-13}$ relative accuracy in double precision
  floating point arithmetic.
\end{abstract}



\section{Introduction} 
\label{sec1}

We consider the computation of the parabolic cylinder function $U(a,z)$, which is a solution of the homogeneous equation
\begin{equation}  \label{01}
\frac{d^{2}y}{dz^{2}}-\left(\frac{1}{4}z^{2}+a \right)y=0,
\end{equation}
and has the integral representations \cite[Eqs. 12.5.1 and 12.5.4]{NIST:DLMF}
\begin{equation}
\label{02}
U(a,z)=\frac{e^{-\frac{1}{4}z^{2}}}{\Gamma
\left(a+\frac{1}{2} \right)}
\int_0^\infty {t^{a-\frac{1}{2}}
e^{-\frac{1}{2}t^{2}-zt}dt}
\quad \left(\Re(a)>-\tfrac{1}{2} \right),
\end{equation}
and
\begin{equation}
\label{03}
U(-a,z)=\sqrt {\frac{2}{\pi }} e^{
\frac{1}{4}z^{2}}\int_0^\infty {t^{a-\frac{1}{2}}e^{-\frac{1}{2}t^{2}}\cos \left( {zt-\tfrac{1}{2}\pi a+\tfrac{1}{4}\pi } \right)dt} \quad \left(\Re(a)>-\tfrac{1}{2} \right).
\end{equation}
It can also be expressed in terms of the confluent hypergeometric function \cite[Chap. 7, Ex. 10.4]{Olver:1997:ASF} by
\begin{equation}
\label{04}
U(a,z)=2^{-\frac{1}{4}-\frac{1}{2}a}
e^{-\frac{1}{4}z^{2}} U\left( \tfrac{1}{2}a+\tfrac{1}{4},
\tfrac{1}{2},\tfrac{1}{2}z^{2} \right).
\end{equation}

As $z\to \infty$ it has the behavior
\begin{equation}
\label{05}
U(a,z)\sim z^{-a-\frac{1}{2}}
e^{-\frac{1}{4}z^{2}}
\quad \left( {\vert \arg (z)\vert \le \tfrac{3}{4}\pi -\delta } 
\right),
\end{equation}
and as such it is the unique solution of (\ref{01}) that is recessive in the sector $|\arg(z)|\leq \pi /4$, since all other independent solutions are exponentially large in this sector.

An important connection formula is given by \cite[Eq. 12.2.19]{NIST:DLMF}
\begin{equation}
\label{i3}
U(a,z)=-ie^{- a \pi i}U(a,-z)+\frac{\sqrt{2\pi}}{
\Gamma\left(a+\tfrac{1}{2}\right)}e^{(\frac{1}{4}-\frac{1}{2}a)\pi i}U(-a,-iz).
\end{equation}

Parabolic cylinder functions have a number of important mathematical and physical applications, including in the study of the Helmholtz equation \cite[Sect. 12.17]{NIST:DLMF}, the uniform asymptotic approximation of solutions of differential equations having two coalescing turning points \cite{Olver:1975:SOL}, and in approximating contour integrals having a coalescing saddle point and an algebraic singularity \cite[Chap. 22]{Temme:2015:AMF}. 

For numerical methods of evaluating parabolic cylinder functions for real $a$ and $z$ see \cite{Gil:2004:IRC}, \cite{Gil:2006:RPC}, \cite{Gil:2006:CRPC}, \cite{Gil:2011:A914}, \cite{Gil:2011:FAC}. In this paper we apply the new results of \cite{Dunster:2021:UAP} for moderately large $|a|$, complemented with stable integral representations, Maclaurin series and a Poincar\'e asymptotic expansion. With this, we describe how the function $U(a,z)$ can be 
accurately computed for real $a$ and complex $z$, filling a gap in the currently available methods of numerical computation.

\section{Methods of computation}

In this section we describe in detail the methods of computation for $U(a,z)$ 
for real $a$ and complex $z$, aiming at computing
the function for real $a$ and complex $z$ with approximately $15$ correct digits.

We can restrict the computation for values of $z$ in the first quadrant on account of the following. Firstly, by Schwarz reflection formula we have that $U(a,\overline{z})=\overline{U(a,z)}$ and therefore we only need to consider the computation for $\Im z\ge 0$. 
Furthermore, considering the connection formula (\ref{i3}), we can compute $U(a,z)$ for $\pi/2 \le \arg z < \pi$ via
 \begin{equation}
 \label{con}
U(a,z)=-ie^{- a \pi i}\overline{U(a,z_1)}+\frac{\sqrt{2\pi}}{
\Gamma\left(a+\tfrac{1}{2}\right)}e^{(\frac{1}{4}-\frac{1}{2}a)\pi i}U(-a,z_2).
\end{equation}
with $z_1=-\overline{z}$ and $z_2=-iz$, both in the first quadrant. Observe that $\arg(z_1)+\arg(z_2)=\pi/2$
and therefore, because $U(a,z)$ is only recessive if $|\arg z|<\pi /4$,  
one of the terms in (\ref{con}) is recessive as $|z|\rightarrow +\infty$ and the other one is dominant for $z$ in the second quadrant, except
along the ray $\arg z=3\pi/4$. Therefore, we do not expect cancellations between terms of similar magnitude
except close to this ray. We notice that there is an infinite number of zeros
of $U(a,z)$ approaching this ray as $|z|\rightarrow +\infty$ \cite{NIST:DLMF}; loss of relative
accuracy is, of course, unavoidable close to the zeros.

In this section, we assume that the methods are implemented in an arbitrary precision environment, and 
that all the functions involved in the computation (including Airy functions) can be computed with arbitrary accuracy. Later, in section  
\ref{numerical} we will describe double precision implementations in Matlab, for which the reachable accuracy will be more limited, in 
part due to loss of accuracy in the computation of Airy functions (the attainable accuracy being $5\times 10^{-13}$ in this case).

In the case of the asymptotic methods (both the uniform asymptotics we describe next and the standard asymptotics for large $|z|$), 
it is important to bound the errors as the methods are not convergent; the number of terms
must be conveniently fixed for accuracy and the region where the expansions can be used must be determined. For 
the convergent methods (integral representation and Maclaurin series) such analysis is not so esssential, 
and we leave the description of the regions of computation for later, when double precision Matlab implementations are discussed.

\subsection{Asymptotic expansions in terms of Airy functions}
\label{sec2}
We use the standard notation for Airy functions of complex argument 
$\mathrm{Ai}_{l}(z):=\mathrm{Ai}(z e^{-2\pi il/3})$ ($l=0,\pm 1$), and recap the main results found in \cite{Dunster:2021:UAP}. 

Firstly, we define

\begin{equation}
\label{19}
w_{l}(u,\tilde{z}) =\mathrm{Ai}_{l}\left( u^{2/3}\zeta \right) 
\mathcal{A}(u,\tilde{z}) +\mathrm{Ai}_{l}^{\prime }\left(u^{2/3}\zeta\right) \mathcal{B}(u,\tilde{z}) \quad (l=0,\pm 1),
\end{equation}
where 
\begin{equation}
\label{xi}
\frac{2}{3}\zeta^{3/2}=\xi=\frac{1}{2}\tilde{z}\sqrt {\tilde{z}^{2}-1}-\frac{1}{2}\ln\left( 
\tilde{z}+\sqrt {\tilde{z}^{2}-1} \right).
\end{equation}
Here principal branches apply. Thus $\xi$ is a continuous function of $\tilde{z}$ in the plane having a 
branch cut along the interval $(-\infty,1]$, and $\xi \geq 0$ for $1\leq \tilde{z}<\infty$. Note $0\leq \zeta<\infty$ 
for $1\leq \tilde{z}<\infty$. The coefficient functions $\mathcal{A}(u,\tilde{z})$ and 
$\mathcal{B}(u,\tilde{z})$ are slowly-varying functions of complex $\tilde{z}$ and positive $u$ in certain parts of the 
$\tilde{z}$ plane, and can be asymptotically evaluated therein by expansions for large $u$ (see (\ref{20}) and (\ref{21}) below).

We remark that $\zeta$ (unlike $\xi$) is an analytic function of $\tilde{z}$ for $\Re(\tilde{z}) \geq 0$, 
in particular at $\tilde{z}=1$. For $-1<\tilde{z}<1$ we observe that $\zeta$ is negative and that
\begin{equation}
\label{zeta1}
\tfrac{2}{3}(-\zeta)^{3/2}
=\tfrac{1}{2}\arccos(\tilde{z})-\tfrac{1}{2}\tilde{z}\sqrt{1-\tilde{z}^{2}}.
\end{equation}

Now define the coefficient functions $\mathcal{A}(u,\tilde{z})$ 
and $\mathcal{B}(u,\tilde{z})$ implicitly via the pair of \emph{exact} expressions
\begin{equation}
\label{19a}
U\left(-\tfrac{1}{2}u,\sqrt {2u} \, \tilde{z}\right)
=\pi^{1/4}u^{-1/12}
\sqrt {2\Gamma \left( \tfrac {1}{2}u+\tfrac{1}{2}\right) } w_{0}(u,\tilde{z}),
\end{equation}
and
\begin{equation}
\label{19b}
U\left(\tfrac{1}{2}u, i\sqrt {2u} \, \tilde{z}\right)
=\frac{2\pi^{3/4}e^{-(3u+1)\pi i/12}
w_{-1}(u,\tilde{z})}{u^{1/12}
\sqrt {\Gamma \left( \tfrac {1}{2}u+\tfrac{1}{2}\right) } }.
\end{equation}
Next, from \cite[Eq. (22)]{Dunster:2020:ASI},
\begin{equation}
\label{eq52}
w_{0}(u,\tilde{z})+e^{-2\pi i/3}w_{1}(u,\tilde{z})+e^{2\pi i/3}w_{-1}(u,\tilde{z})=0,
\end{equation}
and from (\ref{i3})
\begin{equation}
\label{26l}
U\left(a, -i\sqrt {2u} \, \tilde{z}\right)
=ie^{a \pi i}U\left(a, i\sqrt {2u} \, \tilde{z}\right)
+\frac{\sqrt{2\pi}\,
e^{(\frac{1}{2}a-\frac{1}{4})\pi i}}
{\Gamma\left(a+\tfrac{1}{2}\right)}U\left(-a, \sqrt {2u} \, \tilde{z}\right).
\end{equation}
From these we deduce that
\begin{equation}
\label{19c}
U\left(\tfrac{1}{2}u, -i\sqrt {2u} \, \tilde{z}\right)
=\frac{2\pi^{3/4}e^{(3u+1)\pi i/12}
w_{1}(u,\tilde{z})}{u^{1/12}
\sqrt {\Gamma \left( {\tfrac {1}{2}}u+\tfrac{1}{2}\right) } }.
\end{equation}

We now eliminate $\mathcal{B}(u,\tilde{z})$ from (\ref{19a}) and (\ref{19c}) to get the \emph{exact} expression
\begin{multline}
\mathcal{A}(u,\tilde{z})=
\frac{\sqrt{2}\,\pi^{3/4}e^{-5\pi i/6}u^{1/12}}
{\sqrt {\Gamma \left( \tfrac {1}{2} u+\tfrac{1}{2}\right) }}
U\left(-\tfrac{1}{2}u,\sqrt {2u} \, \tilde{z}\right)
\mathrm{Ai}_{1}^{\prime }\left(u^{2/3}\zeta\right)
\\
-\pi^{1/4} u^{1/12}
e^{-(u+1)\pi i/4}
\sqrt {\Gamma \left( \tfrac {1}{2} u+\tfrac{1}{2}\right) }
U\left(\tfrac{1}{2}u, -i\sqrt {2u} \, \tilde{z}\right)
\mathrm{Ai}^{\prime }\left(u^{2/3}\zeta\right),
\label{Aexact}
\end{multline}
and, equivalently, from eliminating $\mathcal{B}(u,\tilde{z})$ from (\ref{19b}) and (\ref{19c})
\begin{multline}
\mathcal{A}(u,\tilde{z})=
\pi^{1/4}u^{1/12}
\sqrt {\Gamma \left( \tfrac {1}{2} u+\tfrac{1}{2}\right) } \left[
e^{(3u-5)\pi i/12}U\left(\tfrac{1}{2}u, i\sqrt {2u} \, \tilde{z}\right)
\mathrm{Ai}_{1}^{\prime }\left(u^{2/3}\zeta\right) \right.
\\
\left. +e^{(5-3u)\pi i/12}U\left(\tfrac{1}{2}u, -i\sqrt {2u} \, \tilde{z}\right)
\mathrm{Ai}_{-1}^{\prime }\left(u^{2/3}\zeta\right) \right].
\label{Aexact1}
\end{multline}

The representation (\ref{Aexact}) will be used to verify the accuracy of the subsequent asymptotic approximation for $\mathcal{A}(u,\tilde{z})$ that we shall use for $\Re(\tilde{z}) \geq 0$. Note that in this equation $i$ can be replaced by $-i$, along with $\mathrm{Ai}_{1}^{\prime }(u^{2/3}\zeta)$ replaced by $\mathrm{Ai}_{-1}^{\prime }(u^{2/3}\zeta)$, to obtain a third representation.

We shall not use (\ref{Aexact1}) except to note that it verifies that $\mathcal{A}(u,\tilde{z})$ is real for positive $\tilde{z}$, an important property that allows us to use the Schwarz reflection principle to only have to compute this function in the first quadrant to have it computable for $\Re(\tilde{z}) \geq 0$. This is because we shall use (\ref{19b}) for $\Re(\tilde{z}) \geq 0$ to compute $U(a,z)$ for $a\geq 10$ and $\Im(z) \geq 0$. For $\Re(\tilde{z}) < 0$ ($\Im(z) < 0$) we can simply use the above expansion and the Schwarz symmetry relation $U(a,z)=\overline{U(a,\overline{z})}$.

Let us now give the asymptotic expansions for $\mathcal{A}(u,\tilde{z})$ 
and $\mathcal{B}(u,\tilde{z})$ that were derived in \cite{Dunster:2021:UAP}. Firstly define
\begin{equation}
\label{06}
\beta=\frac{\tilde{z}}{\sqrt{\tilde{z}^{2}-1}},
\end{equation}
where the principal branch of the square root is taken, so that $\beta$ is positive for $\tilde{z}>1$ 
and is continuous in the plane having a cut along $[-1,1]$. Thus $\beta \rightarrow 1$ as $\tilde{z} \rightarrow \infty$ in any direction.

Next define
\begin{equation}
\label{07}
\mathrm{E}_{1}(\beta)=\tfrac{1}{24}\beta
\left(5\beta^{2}-6\right),
\end{equation}
\begin{equation}
\label{08}
\mathrm{E}_{2}(\beta)=
\tfrac{1}{16}\left(\beta^{2}-1\right)^{2} 
\left(5\beta^{2}-2\right),
\end{equation}
and for $s=2,3,4\cdots$
\begin{equation}
\label{09}
\mathrm{E}_{s+1}(\beta) =
\frac{1}{2} \left(\beta^{2}-1 \right)^{2}\mathrm{E}_{s}^{\prime}(\beta)
+\frac{1}{2}\int_{\sigma(s)}^{\beta}
\left(p^{2}-1 \right)^{2}
\sum\limits_{j=1}^{s-1}
\mathrm{E}_{j}^{\prime}(p)
\mathrm{E}_{s-j}^{\prime}(p) dp,
\end{equation}
where $\sigma(s)=1$ for $s$ odd and $\sigma(s)=0$ for $s$ even. We remark that $\bar{\mathrm{E}}_{2s}(\bar{\beta})$ is even, $\bar{\mathrm{E}}_{2s+1}(\bar{\beta})$ is odd, and $\bar{\mathrm{E}}_{2s}(\pm 1)=0$.

We define two sequences $\left\{a_{s}\right\} _{s=1}^{\infty}$ and $\left\{\tilde{a}_{s}\right\} _{s=1}^{\infty}$ by $a_{1}=a_{2}=\frac{5}{72}$, $\tilde{a}_{1}=\tilde{a}_{2}=-\frac{7}{72}$, with subsequent terms $a_{s}$ and $\tilde{{a}}_{s}$ ($s=2,3,\cdots $) satisfying the same recursion formula
\begin{equation}
\label{16}
b_{s+1}=\frac{1}{2}\left(s+1\right) b_{s}+\frac{1}{2}
\sum\limits_{j=1}^{s-1}{b_{j}b_{s-j}}.
\end{equation}

Then let
\begin{equation}
\label{17}
\mathcal{E}_{s}(\tilde{z}) =\mathrm{E}_{s}(\beta) +
(-1)^{s}a_{s}s^{-1}\xi^{-s},
\end{equation}
and
\begin{equation}
\label{18}
\tilde{\mathcal{E}}_{s}(\tilde{z}) =\mathrm{E}_{s}(\beta)
+(-1)^{s}\tilde{a}_{s}s^{-1}\xi^{-s}.
\end{equation}

For $\Re(\tilde{z}) \geq 0$ and $u \rightarrow \infty$ the coefficient functions in (\ref{19}) then possess the asymptotic expansions
\begin{equation}
\label{20}
\mathcal{A}(u,\tilde{z}) \sim \left( \frac{\zeta }{\tilde{z}^2-1 }\right) ^{1/4}\exp \left\{ \sum\limits_{s=1}^{\infty}\frac{
\tilde{\mathcal{E}}_{2s}(\tilde{z}) }{u^{2s}}\right\} \cosh \left\{ \sum\limits_{s=0}^{\infty}\frac{\tilde{\mathcal{E}}_{2s+1}(\tilde{z}) }{u^{2s+1}}\right\},
\end{equation}
and
\begin{equation}
\label{21}
\mathcal{B}(u,\tilde{z}) \sim 
\frac{1}{u^{1/3}\left\{\zeta \left(\tilde{z}^2-1 \right) \right\}^{1/4}}
\exp \left\{ \sum\limits_{s=1}^{\infty}\frac{
\mathcal{E}_{2s}(\tilde{z}) }{u^{2s}}\right\} \sinh \left\{ \sum\limits_{s=0}^{\infty}\frac{\mathcal{E}_{2s+1}(\tilde{z}) }{u^{2s+1}}\right\}.
\end{equation}
Here it is understood that first sums in (\ref{20}) and (\ref{21}) are zero if $m=0$. Principal branches are taken for the roots, and both coefficient functions are real for $-1<\tilde{z}< \infty$.

On expanding $\mathcal{A}(u,\tilde{z})$ in the traditional asymptotic expansion we get from (\ref{20})
\begin{equation}
\label{22a}
\mathcal{A}(u,\tilde{z}) \sim
\sum\limits_{s=0}^{\infty}\frac{\hat{\mathrm{A}}_{s}(\tilde{z}) }{u^{2s}}
\quad (u \rightarrow \infty),
\end{equation}
where 
\begin{equation}
\label{22b}
\hat{\mathrm{A}}_{s}(\tilde{z})
=\left( \frac{\zeta }{\tilde{z}^2-1 }\right)^{1/4}
\mathrm{A}_{s}(\tilde{z}) \quad (s=0,1,2,\cdots),
\end{equation}
in which
$\mathrm{A}_{0}(\tilde{z})=1$,
\begin{equation}
\label{22c}
\mathrm{A}_{1}(\tilde{z})
=\tfrac{1}{2}\left\{\tilde{\mathcal{E}}_{1}^{2}(\tilde{z})
+2\tilde{\mathcal{E}}_{2}(\tilde{z})\right\},
\end{equation}
\begin{equation}
\label{22d}
\mathrm{A}_{2}(\tilde{z})
=\tfrac{1}{24}\left\{\tilde{\mathcal{E}}_{1}^{4}(\tilde{z})
+12\tilde{\mathcal{E}}_{1}^{2}(\tilde{z})\tilde{\mathcal{E}}_{2}(\tilde{z})
+24\tilde{\mathcal{E}}_{1}(\tilde{z})\tilde{\mathcal{E}}_{3}(\tilde{z})
+12\tilde{\mathcal{E}}_{2}^{2}(\tilde{z})
+24\tilde{\mathcal{E}}_{4}(\tilde{z})\right\},
\end{equation}
and so on. The singularities of $\tilde{\mathcal{E}}_{s}(\tilde{z})$ at $\tilde{z}=1$ 
cancel out in each of these coefficients, rendering them analytic at the turning point.

Similarly 
\begin{equation}
\label{22e}
\mathcal{B}(u,\tilde{z}) \sim
\frac{1}{u^{4/3}}
\sum\limits_{s=0}^{\infty}\frac{
\hat{\mathrm{B}}_{s}(\tilde{z}) }{u^{2s}}
\quad (u \rightarrow \infty),
\end{equation}
where
\begin{equation}
\label{22f}
\hat{\mathrm{B}}_{s}(\tilde{z})
=\left\{\zeta \left(\tilde{z}^2-1 \right) \right\}^{-1/4}
\mathrm{B}_{s}(\tilde{z})  \quad (s=0,1,2,\cdots),
\end{equation}
with the first two terms being $\mathrm{B}_{0}(\tilde{z})=\mathcal{E}_{1}(\tilde{z})$ and
\begin{equation}
\label{22g}
\mathrm{B}_{1}(\tilde{z})=
\tfrac{1}{6}\left\{
\mathcal{E}_{1}^{3}(\tilde{z})
+6\mathcal{E}_{1}(\tilde{z})\mathcal{E}_{2}(\tilde{z})
+6\mathcal{E}_{3}(\tilde{z})
\right\}.
\end{equation}

The expansions (\ref{22a}) and (\ref{22e}) are easy to compute and, for suitable large $u$ (as described in \cref{sec2.1} below), highly accurate in the right half plane, except near $\tilde{z}=1$ where large cancellations occur. As described in \cite{Dunster:2017:COA} we instead use Cauchy integral representations in this instance. Thus we use
\begin{equation}
\label{22h}
\mathcal{A}(u,\tilde{z}) \sim
\frac{1}{2\pi i }
\sum\limits_{s=0}^{\infty}
\frac{1}{u^{2s}}
\oint_{|t-1|=r_{0} } 
\frac{\hat{\mathrm{A}}_{s}(t) 
dt}{t-\tilde{z}},
\end{equation}
and
\begin{equation}
\label{22i}
\mathcal{B}(u,\tilde{z}) \sim
\frac{1}{2\pi i u^{4/3} }
\sum\limits_{s=0}^{\infty}
\frac{1}{u^{2s}}
\oint_{|t-1|=r_{0} } 
\frac{\hat{\mathrm{B}}_{s}(t) 
dt}{t-\tilde{z}},
\end{equation}
where $r_{0}$ is an arbitrary positive number that is less than 2.

We can store values of these coefficients on the contour (using (\ref{22a}) and (\ref{22e})), along with trapezoidal quadrature, allowing rapid computation of $\mathcal{A}(u,\tilde{z})$ and $\mathcal{B}(u,\tilde{z})$ for varying $\tilde{z}$ and $u$ when 
$\tilde{z}$ is close to or equal to the value $1$. We find the optimal choice for the radius of both loops to be $r_{0}=1$.

In order to avoid wrong branches when computing $\zeta$, we recast (\ref{xi}) and (\ref{zeta1}) into a form we shall use, namely
\begin{equation}
\label{23}
\zeta=\tilde{z}^{4/3} \left[\frac {3}{4}\left\{
\sqrt {1-\frac {1}{\tilde{z}^2}}-{\frac {1}{\tilde{z}^{2}}\ln 
\left( 1+\sqrt {1-\frac {1}{\tilde{z}^2}} \right) }-\frac {\ln (\tilde{z}) }{\tilde{z}^{2}} \right\}\right]^{2/3},
\end{equation}
for $|\tilde{z}| \geq 1$. For $|\tilde{z}| < 1$ we use (\ref{zeta1}) to compute $\zeta$, giving the expression
\begin{equation}
\label{24}
\zeta=-\left[\tfrac {3}{4}\left\{
\arccos(\tilde{z})-\tilde{z}\sqrt{1-\tilde{z}^{2}}\right\}\right]^{2/3}.
\end{equation}

In both (\ref{23}) and (\ref{24}) principal branches are taken for all multi-valued terms, and in particular for the inverse cosine we have

\begin{equation}
\label{25}
\arccos(\tilde{z}) =-i\ln  \left( \tilde{z}+i\sqrt {1-\tilde{z}^{2}} \right).
\end{equation}

For for $|\tilde{z}| \geq 1$ ($\tilde{z} \neq 1$) we likewise compute $\beta$ by
\begin{equation}
\label{26}
\beta=\frac{1}{\sqrt{1-\tilde{z}^{-2}}},
\end{equation}
and for $|\tilde{z}|<1$ and $\Im(\tilde{z}) \gtrless 0$
\begin{equation} 
\label{27}
\beta=\mp i \frac{\tilde{z}}{\sqrt{1-\tilde{z}^{2}}}.
\end{equation}

\subsubsection{Error bounds} \label{sec2.1}
Since the asymptotic expansions are not convergent we need to bound the errors. Here we assume the Airy routine's accuracy is known. Consider for fixed positive integer $n$ the error
\begin{equation}
\Delta_{n}(u,\tilde{z})=
\mathcal{A}(u,\tilde{z}) -
\sum\limits_{s=0}^{n}\frac{\hat{\mathrm{A}}_{s}(\tilde{z}) }{u^{2s}}.
\label{Delta}
\end{equation}

\begin{figure}[htbp]
 \centering
 \includegraphics[trim={0 100 0 100},
 width=0.7\textwidth,keepaspectratio]{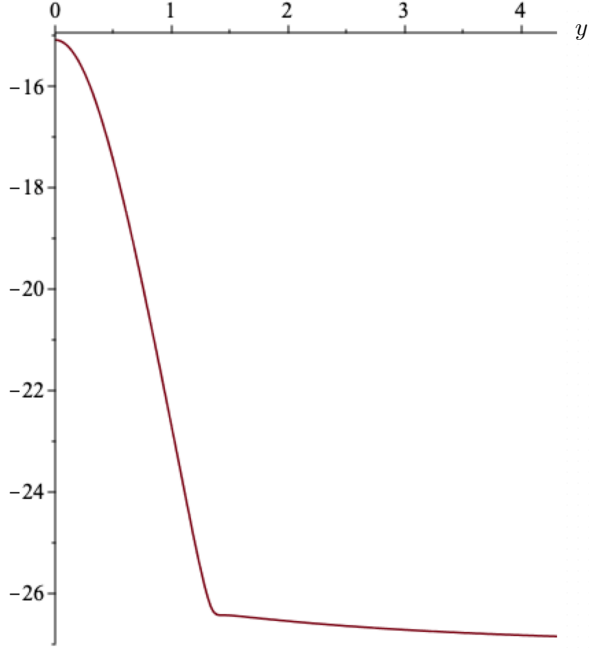}
 \caption{$\log_{10}(|\Delta_{16}(20,iy)|$}
 \label{fig:fig1}
\end{figure}

This is analytic in the right half plane. Moreover, by the Schwarz reflection principle we have $\overline{\Delta_{n}(u,\tilde{z})}
=\Delta_{n}(u,\bar{\tilde{z}})$, and also one can show from the more general error bounds that $\Delta_{n}(u,\tilde{z})\rightarrow 0$ as $\tilde{z} \rightarrow \infty$ along any ray 
$\arg(\tilde{z})=\theta$ where $\theta \in [-\pi/2,\pi/2]$ is fixed. Hence by the maximum modulus theorem the maximum of $|\Delta_{n}(u,\tilde{z})|$ for $\Re(\tilde{z})\geq 0$ occurs on the positive imaginary axis 
$\tilde{z}=i y$, $0 \leq y < \infty$.

In \cref{fig:fig1} we graph $\log_{10}(|\Delta_{16}(20,iy)|$ for $0\leq y \leq 4$ by using (\ref{Aexact}) to compute the exact value of $\mathcal{A}(u,\tilde{z})$ in (\ref{Delta}). 
Observe that the maximum occurs at $y=0$. Thus
\begin{equation}
|\Delta_{16}(20,\tilde{z})|\leq |\Delta_{16}(20,0)|=8.2195\cdots \times 10^{-16},
\end{equation}
for all $\Re(\tilde{z})\geq 0$. Since $\Delta_{n}(u,\tilde{z})$ is monotonically decreasing 
for fixed $n$ and increasing $u$ we see that this bound holds for all $u \geq 20$ and all $\tilde{z}$ in the right half plane, with $n=16$ taken for all values of $u$ and $\tilde{z}$. Of course if $u$ is large we can take a smaller value of $n$ and preserve accuracy, but $n=16$ guarantees accuracy of (\ref{22a}) to at least 15 digits.

Therefore, assuming that the Airy functions and the 
integrals (\ref{22h}) and (\ref{22i}) can be evaluated with sufficient accuracy, we conclude that $U(a,z)$ can be computed with asymptotics 
for $10 \leq a < \infty$, $0 \leq |z| < \infty$ with at least 15 digits.




\subsection{Integral representation}

We start from the integral representation \cite[Eq. 12.5.6]{NIST:DLMF}, which we write as follows:
\begin{equation}
\label{i1}
\begin{array}{l}
U(a,z)=
\displaystyle\frac{e^{\frac{1}{4}z^{2}}}{i\sqrt{2\pi}}\displaystyle\int_{c-i\infty}^{c +i\infty}e^{\phi(t)}dt \quad (c>0,\, 
|\arg(t)|< \tfrac{1}{2}\pi),
\end{array}
\end{equation}
where
\begin{equation}
\label{i1b}
\phi(t)=\tfrac12 t^2 -zt -\alpha \log t,
\end{equation}
and
\begin{equation}
\label{i1c}
\alpha=a+\tfrac12.
\end{equation}
This integral representation, though convergent for all complex values of $\alpha$, is not numerically stable in its present form, particularly
for large $z$. This is observed both from the $z$ dependence in (\ref{i1b}) as well as in the front factor 
$e^{z^2 /4}$ in (\ref{i1}), which does not correspond to the actual behaviour of the function $U(a,z)$ for $|\arg(z)|<3\pi/4$ (see (\ref{05})).

For real $z$, paths of steepest descent were described in \cite{Gil:2004:IRC}. In particular, for $\alpha>0$ and $z \ge 0$
the path of steepest descent crosses the positive saddle point in the direction parallel to the imaginary axis. This motivated the more simple approach considered in \cite{Gil:2006:CRPC} of taking as path of integration a straight line parallel to the imaginary axis and crossing the positive saddle point. Here, we consider a similar approach and take a vertical path through the saddle
in the half plane $\Re(z)>0$. The saddle points come from solving $\phi' (t)=0$ and they are $t=(z\pm \sqrt{z^2+4\alpha})/2$. We take
the saddle point 
\begin{equation}
t_0=\tfrac12 (z+\sqrt{z^2+4\alpha}),
\end{equation}
which has positive real part for any complex $\alpha$ if $\Re(z)>0$; in addition,  if $\alpha$ is real and positive we also have $\Re (t_0)>0$ for any complex $z$ except when
$\Re(z)=0$ and $|z|>2\sqrt{\alpha}$.
 
 When 
$\Re (t_0)>0$ we can use (\ref{i1}) taking the integration path
\begin{equation}
t=t_0+i s,\, s\in (-\infty,+\infty),
\end{equation}
and we write
\begin{equation}
\phi (t)=\phi (t_0)-\tfrac12 s^2+i\alpha f(s/t_0),\, f(w)=w+i\log(1+iw).
\end{equation}
With this we have:


\begin{lemma}
For $\Re(z)>0$ the following holds for any $a\in {\mathbb C}$:
\begin{equation}
\label{i2}
U(a,z)=\displaystyle\frac{1}{\sqrt{2\pi}}
e^{-\frac{1}{4}z\sqrt{z^2+4\alpha}}
e^{\alpha/2}t_0^{-\alpha}
\displaystyle\int_{-\infty}^{+\infty}e^{-\frac12 s^2 +i\alpha f(s/t_0)}ds.
\end{equation}
For real positive $\alpha$ the result holds in ${\mathbb C}\setminus \{z: \Re(z)=0,\, |\Im(z)|>2\sqrt{\alpha}\}$.  
\end{lemma}

Oscillations are not completely eliminated in the integrand of (\ref{i2}) because we have not followed the steepest descent contours, but when the path of integration approaches the steepest descent
paths, the oscillations are slow; this is particularly the case for $\Re (z)>0$ and large $|z|$ (in contrast to (\ref{i1})). 

The representation (\ref{i2}) appears to be stable for $\Re(z)>0$ 
except for some values of $z$ close to the imaginary axis and, in particular, as expected, if $\alpha>0$, $z^2+4\alpha<0$ and $z$ is close to the imaginary axis. In such cases, we consider a path of integration slightly shifted to the right in the original variable $t$ in (\ref{i1}) which is the same as considering the substitution $s=\lambda -i\delta$ in (\ref{i2}), 
where $\delta>0$, $\lambda \in (-\infty,+\infty)$. In our implementation we have used $\delta=1$. 
The recursive trapezoidal rule is a good choice for computing these integrals. See section \ref{numerical} for further details.  





\subsection{Power series}

The Airy uniform asymptotic expansions, together with the integral representation, are enough for an accurate computation of $U(a,z)$ for real $a$ and
complex $z$. However, when power series are available, they are an efficient substitute. For instance, for small $|z|$ the numerical evaluation of integral representations slows down, an the Maclaurin series are a convenient replacement. Similarly, for large $|z|$ we can use asymptotic series in powers of $z^{-1}$ instead of the more costly Airy-type expansions or 
the integral representations.

\subsubsection{Maclaurin series}
\label{sec5a}
For small enough $|z|$ we can consider \cite[Eq. 12.4.1]{NIST:DLMF}
\begin{equation}
\label{80}
U(a,z)=U(a,0)u_{1}(a,z)
+U'(a,0)u_{2}(a,z),
\end{equation}
where
\begin{equation}
\label{81}
u_{1}(a,z)=e^{-\frac{1}{4}z^{2}}
\left\{1+(a+\tfrac{1}{2})\frac{z^{2}}{2!}
+(a+\tfrac{1}{2})(a+\tfrac{5}{2})\frac{z^{4}}{4!}+\cdots\right\},
\end{equation}
\begin{equation}
\label{82}
u_{2}(a,z)=e^{-\frac{1}{4}z^{2}}
\left\{z+(a+\tfrac{3}{2})\frac{z^{3}}{3!}
+(a+\tfrac{3}{2})(a+\tfrac{7}{2})\frac{z^{5}}{5!}+\cdots\right\},
\end{equation}
and equivalently
\begin{equation}
\label{83}
u_{1}(a,z)=e^{\frac{1}{4}z^{2}}
\left\{1+(a-\tfrac{1}{2})\frac{z^{2}}{2!}
+(a-\tfrac{1}{2})(a-\tfrac{5}{2})\frac{z^{4}}{4!}+\cdots\right\},
\end{equation}
\begin{equation}
\label{84}
u_{2}(a,z)=e^{\frac{1}{4}z^{2}}\left\{z+(a-\tfrac{3}{2})\frac{z^{3}}{3!}+(a-\tfrac{3}{2})(a-\tfrac{7}{2})\frac{z^{5}}{5!}+\cdots\right\}.
\end{equation}

We could use (\ref{81}) and (\ref{82}) for $|\arg(z)|\leq 3 \pi/4$, and (\ref{83}) and (\ref{84}) otherwise. This is to match the exponential behaviour (\ref{05}) of $U(a,z)$ for large $z$, although of course this is a minor consideration since $z$ is not large.

The series are, of course, convergent for all complex $z$, however they are only interesting numerically 
for small $|z|$, where they converge fast. For example, we have checked that at most 36 terms are required (the least accurate case is $z=\pm 3i$, $a=10$) to give double precision accuracy in the range 
$|z|\le 3$ $|a|\le 10$.

In our numerical implementation, we will add terms of the series for $u_1 (a,z)$ and $u_2 (a,z)$ until
the last term gives a relative contribution less than the required accuracy.

\subsubsection{Large argument asymptotic expansions}
\label{sec6a}

We have for positive integer $n$ \cite[12.9.1]{NIST:DLMF}
\begin{equation}
\label{85}
U(a,z) =
z^{-a-\frac{1}{2}}e^{-\frac{1}{4}z^{2}} \left\{
\sum_{s=0}^{
n-1}(-1)^{s}\frac{{\left(a+\frac{1}{2}\right)_{2s}}}
{s!(2z^{2})^{s}}
+ R_{n}(a,z) \right\}
\quad (|\arg(z)| \leq \tfrac{3}{4}\pi - \delta)
\end{equation}
where
\begin{equation}
\label{85a}
R_{n}(a,z) = \mathcal{O}\left(z^{-2n} \right)
\quad (z \rightarrow \infty, \,|\arg(z)| \leq \tfrac{3}{4}\pi - \delta)
\end{equation}

In this case, differently to Maclaurin series, we can not use as many terms as neeed, because the expansion
is divergent, and it is convenient to bound the error in a certain region, which we do next.

Consider $12 \leq |z| \leq \infty$, $|\arg(z)| \leq \tfrac{1}{2}\pi$ with $n$ and $a$ fixed. By Schwarz symmetry and the maximum modulus theorem the maximum of $|R_{n}(a,z)|$ occurs either on the quarter circle $|z|=12$ lying in the first quadrant, or along the part of the imaginary axis $z=iy$, $12 \leq y < \infty$.

Choose $n=35$. Then for $a=10$ we find this maximum occurs at $z=12i$ with the value $|R_{35}(10,12i)|=1.131\cdots \times 10^{-14}$. For the same $n$ and other $a$ values in $[-10,10)$ we find smaller values of the error term, which indicates that $|R_{35}(a,z)| \leq 1.131\cdots \times 10^{-14}$ for all $a \in [-10,10]$ and $\Re(z)\geq 0$ with $12 \leq |z| \leq \infty$.

Although numerically convincing, it is not so easy to prove rigorously that for fixed $n$ and $z$, with $a \in [-10,10]$, that the maximum of $|R_{n}(a,z)|$ occurs $a=10$. However, by analysing a (Laplace-type) contour integral for $R_{n}(a,z)$ we have the following bound, the proof of which is given in \cref{secB}.
\begin{theorem} \label{thm:thmlargez}
For $0\leq |a| \leq 10$, $12 \leq |z| < \infty$ and $|\arg(z)| \leq \pi/2$
\begin{equation}
\label{106}
\left |R_{35}(a,z)\right| < 6.24 \times 10^{-14}.
\end{equation}
\end{theorem}

For $\Re(z)<0$, $\Im(z)\ge 0$ we can again use (\ref{i3}), and for $\Re(z)<0$, $\Im(z)< 0$ the complex conjugate.

\section{Numerical tests}
\label{numerical}

In this section we implement fixed double precision accuracy versions of the different methods. We choose Matlab as environment for our tests. 
The accuracy and ranges of application of the methods will be more limited than described earlier. There are several new sources of accuracy loss when we compute the functions in fixed precision.

First, of course, we have rounding errors in all the computations, including the elementary evaluations. In the second place, the 
expansions in terms of Airy functions will have more limited accuracy because of the evaluation of these functions; also, the numerical evaluation of the contour integrals (\ref{22h}) and (\ref{22i}) may introduce additional errors. Finally, there is an
unavoidable loss of accuracy for large arguments due to the conditioning of the function $U(a,z)$, which may increase or decrease 
exponentially. This last source of inaccuracy can only be avoided by computing scaled versions, as was done in \cite{Gil:2006:CRPC} 
for the real case, in such a way that the bad conditioning is isolated in a single elementary factor.

As we discussed before, the asymptotic expansions in terms of Airy functions are powerful methods which 
allows one to compute $U(a,z)$ for $|a|>10$ in double accuracy, assuming Airy functions are sufficiently accurate. When we implement these methods in Matlab, where the Airy functions are not so accurately computed, we have to lower the expectations. 

In \cref{asymp}, we represent in the $(|z|,a)$-plane 
the points where $5\times 10^{-13}$ relative accuracy is \emph{not} reached for both types of asymptotic expansions. The expression we used for checking is the recurrence relation 
\cite[12.8.1]{NIST:DLMF}. 
We have limited the figure to the range $|a|<30$ and $|z|<30$ for a better visualization of the regions
of application, and also to avoid unavoidable error degradation due to bad conditioning of $U(a,z)$. 
For this range, we test the recurrence relation for values of $z$ in the principal domain of computation.
More specificially,  we randomly generate values of $|z|\in [0,30]$, for $\arg(z)\in [0,\pi/2]$ and 
$a\in [-30,30]$. We generate $10^5$ points for each of the two methods.

\begin{figure}[htbp]
 \centering
 \includegraphics[width=0.7\textwidth,keepaspectratio]{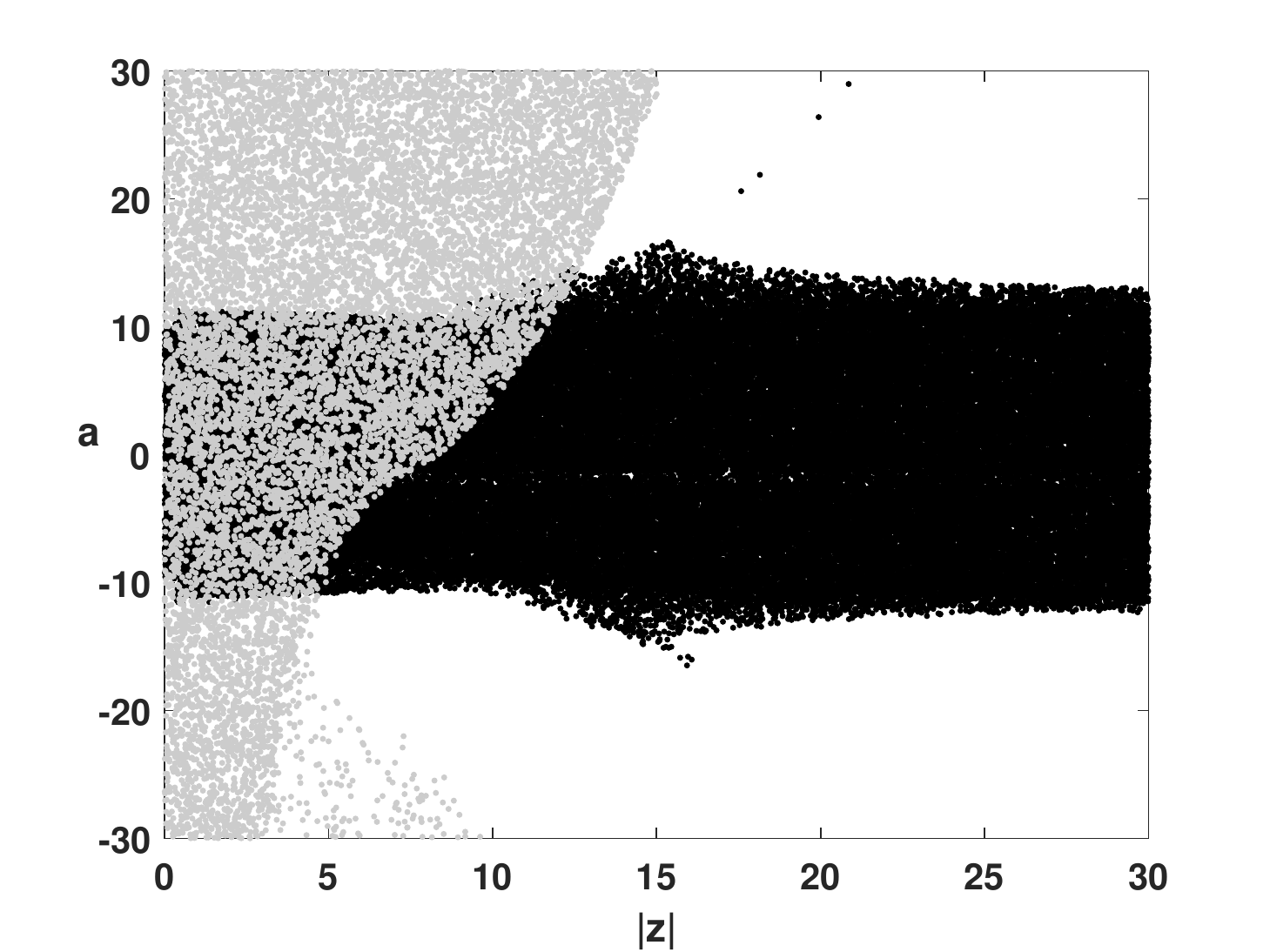}
 \caption{Points where $5\times 10^{-13}$ accuracy is not reached using asymptotic expansions. Black points correspond to Airy-type expansions and gray points to Poincar\'e type}
 \label{asymp}
\end{figure}

The number of terms considered for the Airy-type expansions is as described in \ref{sec2.1}, while for
the Poincar\'e expansion we adopt a more heuristic approach, by adding terms until the last (neglected) term is, 
relative to te accumulated sum, smaller that the accuracy goal; in this process we limit the maximum number of terms to $50$ 
(which is enough, considering Theorem \ref{thm:thmlargez}). In the case of the Airy-type expansions, close to the turning point we
have to compute the coefficients (\ref{22h}) and (\ref{22i}) by contour integration using the trapezoidal rule with $2000$ points.

In \cref{asymp} we observe that the points for $|a|>12$ where $5\times 10^{-13}$ accuracy is not
reached with Airy-type asymptotics can be covered with Poincar\'e asymptotics, from which we conclude
that $U(a,z)$ can be computed with asymptotic expansions for $|a|>12$ with such accuracy. Smaller values
of $|a|$ will be covered by integral representations.

We repeat the same kind of analysis as done in \cref{asymp} for the convergent methods (Maclaurin series and integral representations), and
we show the result in \cref{converg}. As we see, the integral representation produces accurate results for $|a|<20$, while for larger values there is some accuracy loss for small $|z|$. 

\begin{figure}[htbp]
 \centering
 \includegraphics[width=0.7\textwidth,keepaspectratio]{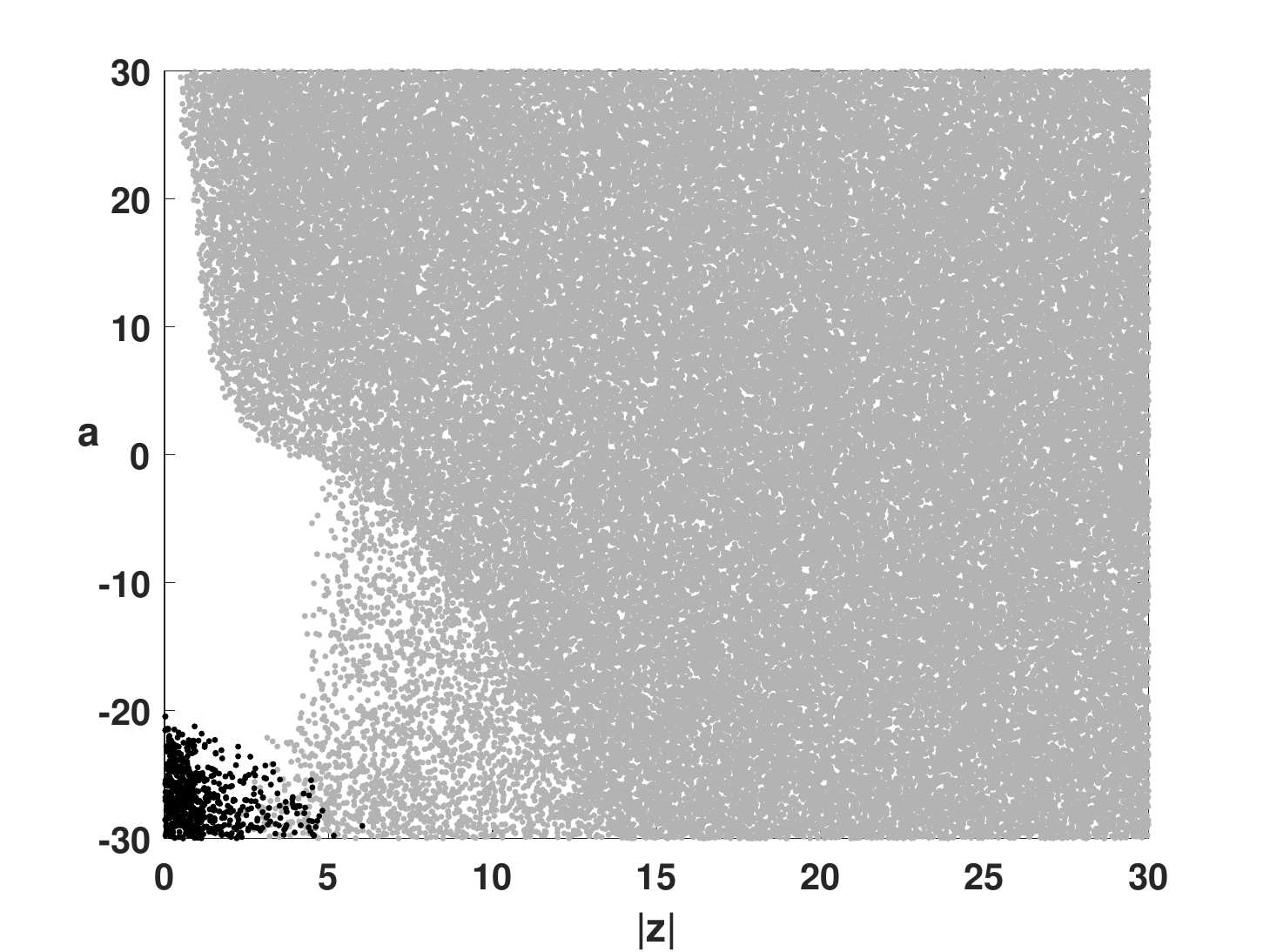}
 \caption{Points where $5\times 10^{-13}$ accuracy is not reached. Black points correspond to the integral representation 
 and gray points to Maclaurin series}
 \label{converg}
\end{figure}

The recursive trapezoidal rule is an efficient method of computation for the integral representations. For computing the integrals, one of course has to
truncate the interval of integration, and we have numerically checked that integrating (\ref{i2}) (or the modified integral with shifting) in $[-15,15]$ is sufficient for our purposes.

From the light areas of the previous two figures, it is clear that it is possible to compute $U(a,z)$ with close to $10^{-13}$ accuracy by combining
asymptotic expansions (both in terms of Airy functions and of Poincar\'e type) with integral representations. For instance, a simple choice of the regions of computation in the principal domain ($0\le \arg z\le \pi/2$) could consist in using asymptotics for large $|z|$ of $|z|>12+\frac{1}{6}|a|$, and in a different case use Airy-type asymptotics if $|a|>20$, and integral representations otherwise. 

As a final check, we have tested this combination of methods, using 
Schwarz reflection and the connection formula (\ref{con}) for $z$ outside the principal domain\footnote{Observe that  when $a+1/2$ is close to a negative integer special precautions should be taken for
computing $\Gamma (a+1/2)$, specially
when $|z|$ is large and $3\pi/4<|\arg z|\le \pi$; the reflection formula for the gamma function helps in
evaluating such values accurately. For $a+1/2\in {\mathbb Z}^{-}$ the second term in this connection formula is zero.}. 
We have randomly generated $10^{6}$ values of 
$|z|\in [0,30]$, of $\arg(z)\in (-\pi, \pi]$ and of $a\in [-30,30]$, and we again test the recurrence relation. With this selection we obtain
a maximum relative error of $4.7 \times 10^{-13}$, and less than at $1\%$ of the $10^6$ points the accuracy was worse than $5\times 10^{-14}$.

\begin{figure}[htbp]
 \begin{minipage}{6cm}
 \includegraphics[width=1.1\textwidth]{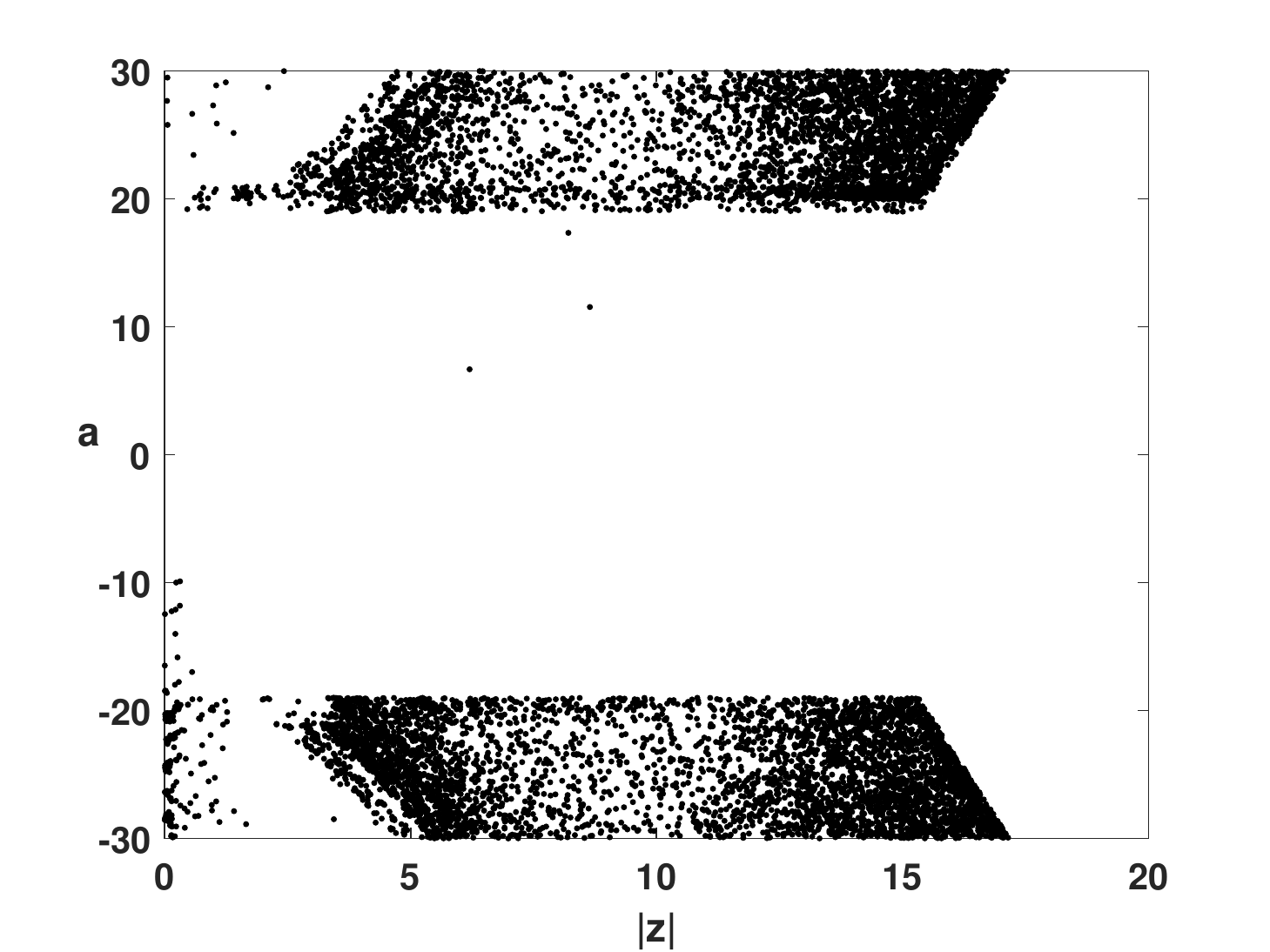}
 \end{minipage}
 \begin{minipage}{6cm}
 \includegraphics[width=1.1\textwidth]{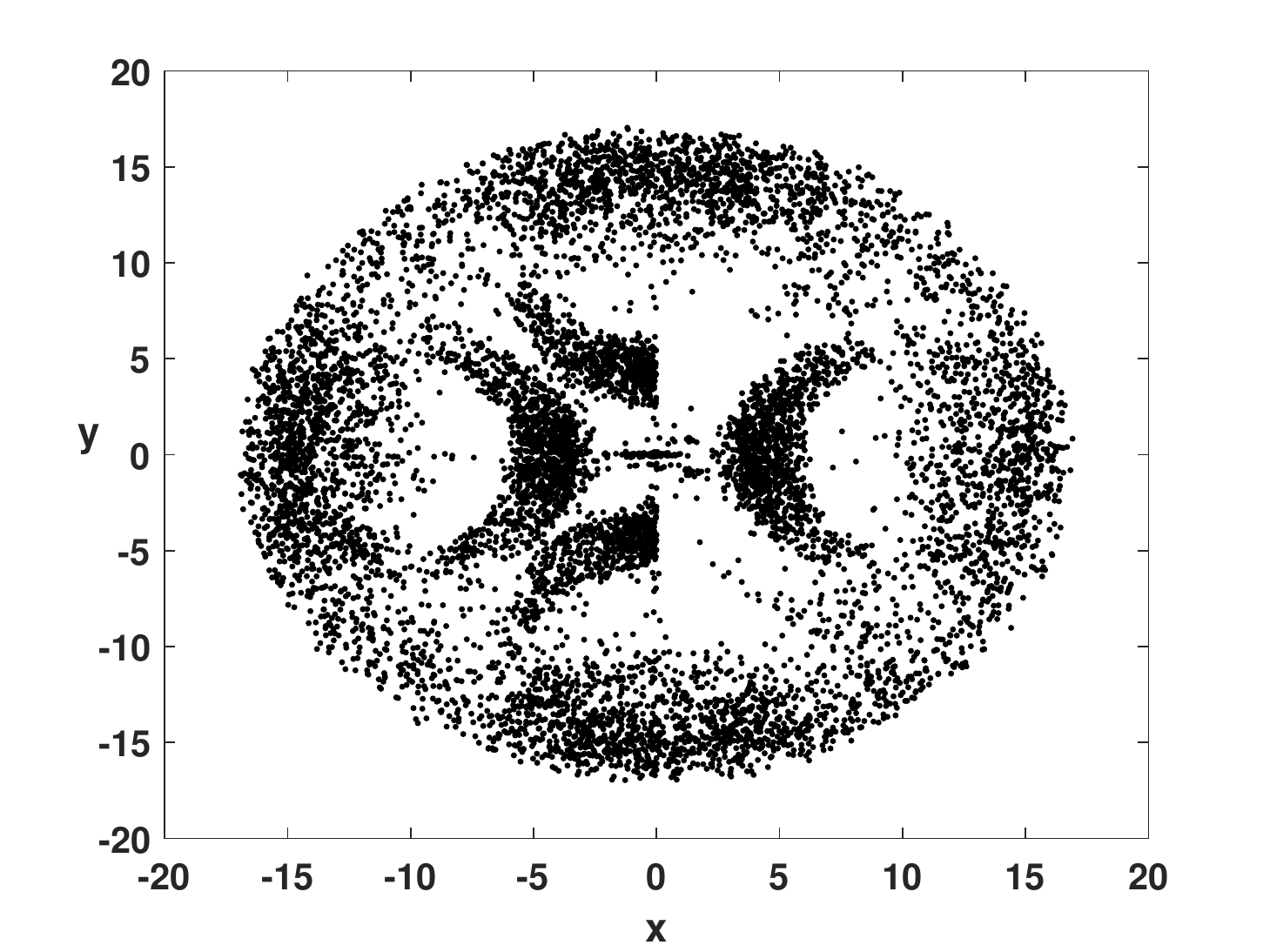}
 \end{minipage}
 \caption{Points where $5\times 10^{-14}$ accuracy is not reached for the combined method. Left: points on the $(|z|,a)$-plane. 
 Right: points in the $z$-plane ($x=\Re z$, $y=\Im z$)}
 \label{algo}
\end{figure}

In \cref{algo} (left) it becomes clear that the largest errors correspond to the Airy-type asymptotics, because it takes place for $|a|>20$. On the other hand, the plot on the right shows that these larger errors are approximatelly limited to $|z|<17$, which points to loss of accuracy due to the computation of Airy functions; Matlab relies on Amos algorithm \cite{Amos:1986:A6A}, which computes Airy functions in terms of modified Bessel functions and appears to give errors with a relative error around $10^{-13}$ for moderate arguments (as shown in \cite{Gil:2002:A8A}); for larger arguments it uses asymptotics and the results could be more accurate. A more accurate implementation of the Airy-type expansions should rely on alternative and more accurate methods of computation. A future implementation using the methods described in \cite{Gil:2002:A8A} and \cite{Gil:2002:CCA} is indeed a possibility.
 
As a final check, we have repeated the analysis, but lowering the value of $|a|$ for which Airy-type asymptotics is used to $|a|>15$. We obtained similar results, but with a slightly worse maximum relative error of $6\times 10^{-13}$, and with $1.4\%$ of the random points giving an accuracy worse than $5\times 10^{-14}$.

\appendix
\section{Proof of \cref{thm:thmlargez}}
\label{secB}

\begin{proof}[Proof of Theorem \ref{thm:thmlargez}]
Assume temporarily that $\Re(a)>-\tfrac{1}{2}$, and use \cite[Sec. 11.2]{Temme:2015:AMF}
\begin{equation}
\label{86}
U(a,z)=
\frac{z^{a+\frac{1}{2}}e^{-\frac{1}{4}z^{2}}}
{\Gamma\left(a+\frac{1}{2} \right)}
\int_0^\infty t^{a-\frac{1}{2}}
e^{-z^2 t} f(a,t) dt
\quad \left(|\arg(z)|<\tfrac{1}{4}\pi \right),
\end{equation}
where
\begin{equation}
\label{87}
f(a,t)=\left\{ {\frac {1}{t} \left( \sqrt {1+2t}-1 \right) }
\right\} ^{a-{\frac{1}{2}}}{\frac {1}{\sqrt {1+2t}}}
\end{equation}
which has the Maclaurin series expansion
\begin{equation}
\label{88}
f(a,t)=\sum_{k=0}^{\infty}(-1)^k
\frac{\left( a+\tfrac{1}{2} \right)_{2k}}
{k!\left( a+\tfrac{1}{2} \right)_{k}}
\left(\frac{t}{2}\right)^{k} \quad (|t| < \tfrac{1}{2})
\end{equation}

Write this as
\begin{equation}
\label{89}
f(a,t)=\sum_{k=0}^{n-1}(-1)^k
\frac{\left( a+\tfrac{1}{2} \right)_{2k}}
{k!\left( a+\tfrac{1}{2} \right)_{k}}
\left(\frac{t}{2}\right)^{k}
+ \eta_{n}(a,t)
\end{equation}
where
\begin{equation}
\label{89a}
\eta_{n}(a,t)=
(-1)^n\frac{\left( a+\tfrac{1}{2} \right)_{2n}}
{n!\left( a+\tfrac{1}{2} \right)_{n}}
\left(\frac{t}{2}\right)^{n}\hat{\eta}_{n}(a,t)
\end{equation}
in which
\begin{equation}
\label{90}
\hat{\eta}_{n}(a,t)=1-\frac{(4n+2a+1)(4n+2a+3) }
{4(n+1)(2n+2a+1)}t
+ \ldots  \quad (|t| < \tfrac{1}{2})
\end{equation}

We get
\begin{equation}
\label{91}
U(a,z) =
z^{-a-\frac{1}{2}}e^{-\frac{1}{4}z^{2}}
\left\{\sum_{s=0}^{n-1}
(-1)^{s}\frac{{\left(a+\frac{1}{2}\right)_{2s}}}
{s!(2z^{2})^{s}}
+ R_{n}(a,z)\right\}
\end{equation}
where
\begin{equation}
\label{92}
R_{n}(a,z)=
\frac{z^{2a+1}}
{\Gamma\left(a+\frac{1}{2} \right)}
\int_0^\infty t^{a-\frac{1}{2}}
e^{-z^2 t} \eta_{n}(a,t) dt
\end{equation}

At this stage we can relax the restriction $\Re(a)>-\tfrac{1}{2}$ to $n+\Re(a)>-\tfrac{1}{2}$, since each coefficient in (\ref{91}) is defined for unrestricted $a$, and on referring to (\ref{89a}) the integral in (\ref{92}) converges under this weaker condition.

Now consider the integral in (\ref{92}). Make the change of variable $t \rightarrow t\exp(-2i \theta)$ where $\theta =\arg(z)$, and deform back to the nonnegative real axis. We have
\begin{equation}
\label{93}
\int_0^\infty t^{a-\frac{1}{2}}
e^{-z^2 t} \eta_{n}(a,t) dt
=e^{-i (2a + 1) \theta}\int_0^\infty t^{a-\frac{1}{2}}
e^{-|z|^2 t} 
\eta_{n}\left(a,t e^{-2i \theta} \right) dt.
\end{equation}
Recalling the definitions (\ref{87}) and (\ref{89}), we see by analytic continuation this certainly holds for $0 \leq \theta =\arg(z) \leq \pi/2$. Now split the integral on the RHS of (\ref{93}) into two, one from $0\leq t \leq 0.45$ and the other from $0.45 \leq t \leq \infty$, and call these $I_{n,1}(a,z)$ and $I_{n,2}(a,z)$, respectively. Thus
\begin{equation}
\label{94a}
R_{n}(a,z)=
\frac{z^{2a+1}}
{\Gamma\left(a+\frac{1}{2} \right)}
\left \{I_{n,1}(a,z)+I_{n,2}(a,z)  \right \}
\end{equation}

For the first one we use from (\ref{89a}) and (\ref{90}) for $0 \leq \theta \leq \pi/2$
\begin{equation}
\label{95}
\left |\eta_{n}\left(a,t e^{-2i \theta)}\right) 
\right | \leq
\frac{\left( a+\tfrac{1}{2} \right)_{2n}}
{n!\left( a+\tfrac{1}{2} \right)_{n}}
\left(\frac{t}{2}\right)^{n}\hat{\eta}_{n}(a,-t) 
\end{equation}
noting that
\begin{equation}
\label{96}
\hat{\eta}_{n}(a,-t)=1+\frac{(4n+2a+1)(4n+2a+3) }
{4(n+1)(2n+2a+1)}t
+ \ldots  \quad (t < \tfrac{1}{2})
\end{equation}
Now by monotonicity
\begin{equation}
\label{97}
\sup_{0 \leq t \leq 0.45}\hat{\eta}_{n}(a,-t)
=\hat{\eta}_{n}(a,-0.45) 
\end{equation}
Next 
\begin{equation}
\label{98}
\left |I_{n,1}(a,z)\right | 
\leq \int_0^{0.45} t^{a-\frac{1}{2}}
e^{-|z|^2 t} \left |
\eta_{n}\left(a,t e^{-2i \theta}\right) \right | dt
\end{equation}
and so for $0 \leq \theta \leq \pi/2$
\begin{multline}
\label{99}
\left |I_{n,1}(a,z)\right | 
\leq \frac{\left( a+\tfrac{1}{2} \right)_{2n}}
{2^{n}n!\left( a+\tfrac{1}{2} \right)_{n}}
\int_0^{0.45} t^{n+a-\frac{1}{2}}
e^{-|z|^2 t} \hat{\eta}_{n}(a,-t) dt \\
\leq
\frac{\hat{\eta}_{n}(a,-0.45) \left( a+\tfrac{1}{2} \right)_{2n}}
{2^{n}n!\left( a+\tfrac{1}{2} \right)_{n}}
\int_0^{0.45} t^{n+a-\frac{1}{2}}
e^{-|z|^2 t}  dt \\
< \frac{\hat{\eta}_{n}(a,-0.45) \left( a+\tfrac{1}{2} \right)_{2n}}
{2^{n}n!\left( a+\tfrac{1}{2} \right)_{n}}
\int_0^{\infty} t^{n+a-\frac{1}{2}}
e^{-|z|^2 t}  dt
=\frac{\hat{\eta}_{n}(a,-0.45) 
\Gamma\left( a+\tfrac{1}{2} \right)\left( a+\tfrac{1}{2} \right)_{2n}}
{2^{n} n!|z|^{2n+2a+1}}
\end{multline}
and hence for the first term on the RHS of (\ref{94a})
\begin{equation}
\label{100}
\frac{|z|^{2a+1}}
{\Gamma\left(a+\frac{1}{2} \right)}
\left |I_{n,1}(a,z)\right |
< \frac{\hat{\eta}_{n}(a,-0.45) 
\left( a+\tfrac{1}{2} \right)_{2n}}
{n!|2z^2|^{n}}
\end{equation}

For the second term one can show from (\ref{87}) and (\ref{89}) that for $\theta \in [0,\pi/2]$
\begin{equation}
\label{101}
\left |\eta_{n}\left(a,t e^{-2i \theta}\right) \right |
\leq h_{n}(a,t)
\end{equation}
where
\begin{equation}
\label{102}
h_{n}(a,t)=
\left\{ {\frac {1}{t} \left | \sqrt {1-2t}-1 \right |}
\right\} ^{a-{\frac{1}{2}}}{\frac {1}{\sqrt {|1-2t|}}}
+\sum_{k=0}^{n-1}
\frac{\left( a+\tfrac{1}{2} \right)_{2k}}
{k!\left( a+\tfrac{1}{2} \right)_{k}}
\left(\frac{t}{2}\right)^{k}
\end{equation}
for $a \geq 1/2$, and
\begin{equation}
\label{103}
h_{n}(a,t)=
\left\{ {\frac {1}{2} \left( \sqrt {1+2t}+1 \right) }
\right\} ^{\frac{1}{2}-a}{\frac {1}{\sqrt{|1-2t|}}}
+\sum_{k=0}^{n-1}
\frac{\left( a+\tfrac{1}{2} \right)_{2k}}
{k!\left( a+\tfrac{1}{2} \right)_{k}}
\left(\frac{t}{2}\right)^{k}
\end{equation}
for $0 \leq a \leq 1/2$. From the definition of $I_{n,2}(a,z)$, namely the RHS of (\ref{93}) with the lower limit $0$ replaced by $0.45$, then we simply have
\begin{equation}
\label{104}
\frac{|z|^{2a+1}}
{\Gamma\left(a+\frac{1}{2} \right)}
\left |I_{n,2}(a,z)\right |
\leq 
\frac{|z|}
{\Gamma\left(a+\frac{1}{2} \right)}
\int_{0.45}^{\infty} \left(z^2 t\right)^{a-\frac{1}{2}}
e^{-|z|^2 t} h_{n}(a,t) dt
\end{equation}

We are considering $0\leq a \leq 10$ and $12 \leq |z| < \infty$, and it is straightforward to show the the RHS of (\ref{100}) and (\ref{104}) are maximized at the extreme values $a=10$ and $|z|=12$. 

Taking $n=35$ we then compute for $a=10$ and $|z|=12$ the values $5.66790\cdots \times 10^{-14}$ and $5.95016 \cdots \times 10^{-15}$ respectively for these bounds. Plugging these into (\ref{94a}), referring to (\ref{91}), and using the Schwarz reflection principal, we arrive at (\ref{106}).
\end{proof}

\section*{Acknowledgments}
The authors acknowledge financial support from Ministerio de Ciencia e Innovación, projects
PGC2018-098279-B-I00 (MCIN/AEI/10.13039/ 501100011033/FEDER "Una manera de hacer Europa") and 
PID2021-127252NB-I00 (MCIN/AEI/10.13039/ 501100011033/FEDER, UE).

\bibliographystyle{siamplain}
\bibliography{biblio}

\end{document}